\documentclass{article}
\usepackage{graphicx} % Required for inserting images
\usepackage{amsthm}
\usepackage{amsfonts, amsmath, amssymb, amsthm}
\usepackage{cleveref}
\usepackage[utf8]{inputenc}
\usepackage{url}

\title{The first-order part of the Ordered Variable Word theorem}
\date{\today}

\newtheorem*{statement}{Statement}
\newtheorem{theorem}{Theorem}
\numberwithin{theorem}{section}
\newtheorem{maintheorem}[theorem]{Main Theorem}
\newtheorem{lemma}[theorem]{Lemma}
\newtheorem{question}[theorem]{Question}
\newtheorem{proposition}[theorem]{Proposition}

\newtheorem{definition}[theorem]{Definition}
\newtheorem{corollary}[theorem]{Corollary}

\newtheorem{claim}{Claim}
\def\TT{\mathbb{T}}
\def\PP{\mathbb{P}}

\newcommand{\Psf}{\mathsf{P}}

\def\C{\mathcal{C}}
\def\I{\mathcal{I}}

\def\L{\mathcal{L}}

\def\P{\mathcal{P}}
\def\T{\mathcal{T}}

\def\M{\mathcal{M}}

\def\H{\mathcal{H}}

\newcommand{\uh}{\upharpoonright}
\newcommand{\uhr}{\upharpoonright}

\newcommand{\qvdash}{\operatorname{{?}{\vdash}}}
\newcommand{\nqvdash}{\operatorname{{?}{\nvdash}}}

\newcommand{\ISig}{\mathsf{I}\Sigma^0}

\newcommand{\IPi}{\mathsf{I}\Pi^0}
\newcommand{\IDelta}{\mathsf{I}\Delta^0}

\newcommand{\BSig}{\mathsf{B}\Sigma^0}
\newcommand{\BPi}{\mathsf{B}\Pi^0}

\newcommand{\LSig}{\mathsf{L}\Sigma^0}
\newcommand{\LPi}{\mathsf{L}\Pi^0}
\newcommand{\LDelta}{\mathsf{L}\Delta^0}
\newcommand{\PSig}{\mathsf{P}\Sigma^0}

\newcommand{\Exp}{\mathsf{exp}}

\newcommand{\PAm}{\mathsf{PA}^-}
\newcommand{\GammaCA}[1]{{#1}\mbox{-}\mathsf{CA}}

\newcommand{\RCA}[0]{\mathsf{RCA}}
\newcommand{\WKL}[0]{\mathsf{WKL}}

\newcommand{\RT}[0]{\mathsf{RT}}

\newcommand{\CSL}[0]{\mathsf{CSL}}
\newcommand{\OVW}[0]{\mathsf{OVW}}

\newcommand{\NN}[0]{\mathbb{N}}

\newcommand{\card}{\operatorname{card}} % Set cardinal
\newcommand{\OSub}{\mathsf{OSub}}

\newcommand{\WSub}{\mathsf{WSub}}

\newcommand{\Def}[1]{{#1}\mbox{-}\mathsf{Def}}

\newcommand{\h}[1]{\hat{#1}}

\def\qt#1{``#1''}%

\makeatletter
\newtheorem*{rep@theorem}{\rep@title}
\newcommand{\newreptheorem}[2]{%
\newenvironment{rep#1}[1]{%
 \def\rep@title{#2 \ref{##1}}%
 \begin{rep@theorem}}%
 {\end{rep@theorem}}}
\makeatother

\newreptheorem{theorem}{Theorem}
\newreptheorem{maintheorem}{Main Theorem}

\usepackage{xcolor}

\usepackage{authblk}
\DeclareSymbolFont{bbold}{U}{bbold}{m}{n}
\DeclareMathSymbol{\bbomega}{\mathord}{bbold}{"7F}

\author{Ludovic Patey}

\def\qt#1{``#1''}%

\begin{document}

\maketitle

\begin{abstract}
An ordered $\omega$-variable sequence over a finite alphabet~$A$ is an infinite sequence over the extended alphabet $A \sqcup \{ x_0, x_1, \dots \}$, where each variable kind appears at least once, and the last occurrence of $x_i$ appears before the first occurrence of $x_{i+1}$.
The Ordered Variable Word theorem states that for every finite coloring of the words over a finite alphabet~$A$, there is an ordered $\omega$-variable sequence~$W$ over~$A$ such that every word induced by substituting variables with letters, and cutting before the first undefined variable, has the same color.

In this article, we prove that the Ordered Variable Word theorem is $\Pi^1_1$-conservative over $\RCA_0 + \BSig_2$. It follows that the tree theorem for singletons by Chubb, Hirst, and McNicholl is $\Pi^1_1$-conservative over $\RCA_0 + \BSig_2$, answering a long-standing open question. \Cref{lem:pi01-extension} was found by Claude Opus 5.5 (Anthropic).
\bigskip

\textbf{MSC classes} Primary: 03B30, 03F30 Secondary: 05D10
\end{abstract}

\section{Introduction}

In 1984, Carlson and Simpson~\cite{carlson1984dual} proved a dual form of Ramsey's theorem, using as a combinatorial core a partition theorem about variable words, now known as the Ordered Variable Word theorem for words. A \emph{word} over a finite alphabet~$A$ is a finite sequence $a_0 \cdots a_{n-1}$ of elements of~$A$. We denote by $A^{<\omega}$ the set of words over~$A$. An \emph{ordered $\omega$-variable sequence} over~$A$ is an infinite sequence over the alphabet $A \sqcup \{ x_0, x_1, \dots \}$ such that each variable kind appears at least once, and such that for each~$i \in \NN$, the last occurrence of $x_i$ appears before the first occurrence of $x_{i+1}$. Given an ordered $\omega$-variable sequence $W$ and a word~$u = a_0 \cdots a_{n-1}$, $W[u]$ denotes the word obtained from $W$ by replacing every occurrence of $x_i$ with $a_i$, and cutting before the first occurrence of $x_n$. 

\begin{theorem}[Carlson and Simpson~\cite{carlson1984dual}]\label[theorem]{thm:ovw0}
For every $\ell \geq 1$ and every coloring $f : A^{<\omega} \to \ell$ over a finite alphabet~$A$, there is an ordered $\omega$-variable sequence~$W$ over~$A$ such that the set $\{ W[u] : u \in A^{<\omega} \}$ is $f$-monochromatic.
\end{theorem}

We denote by $\OVW^0$, standing for Ordered Variable Word theorem, the statement of \Cref{thm:ovw0}. Carlson~\cite{carlson1988some} later proved a higher-order generalization of $\OVW^0$ to colorings of finite-dimensional ordered variable words. The Ordered Variable Word  theorem belongs to a family of tree partition theorems~\cite{chubb2009reverse,milliken1979ramsey} with applications in Structural Ramsey Theory~\cite{devlin1980some,hubivcka2025big,sauer2006coloring}.

From a meta-mathematical viewpoint, the Ordered Variable Word theorem was extensively studied in reverse mathematics and computability theory~\cite{miller2004effectiveness,Dzhafarov2017Effectiveness,liu2019computable,angles2023carlson,houerou2024pi}. Recently, Liu and Patey~\cite{liu2026reverse} made significant advances in the understanding of the computability-theoretic features of the Ordered Variable Word theorem, and proved the following weakly low basis theorem:

\begin{theorem}[Liu and Patey~\cite{liu2026reverse}]\label[theorem]{thm:ovw0-weakly-low}
Fix a set~$Z$ and a set~$P$ of PA degree over~$Z'$.
For every $\ell \geq 1$ and every $Z$-computable coloring $f : A^{<\omega} \to \ell$ over a finite alphabet~$A$, there is an ordered $\omega$-variable sequence~$W$ over~$A$ such that the set $\{ W[u] : u \in A^{<\omega} \}$ is $f$-monochromatic and $(W \oplus Z)' \leq_T P$.
\end{theorem}

By an isomorphism theorem for weak K\"onig's lemma~\cite{fiori2026isomorphism}, there is a correspondence between $\Pi^1_1$-conservation theorems over $\RCA_0 + \BSig_2$ and proving formalized weakly low basis theorems over $\RCA_0 + \BSig_2$. In this article, we formalize \Cref{thm:ovw0-weakly-low} over $\RCA_0 + \BSig_2$, and deduce the following main theorem:

\begin{maintheorem}\label[maintheorem]{thm:main1}
$\RCA_0 + \OVW^0$ is $\Pi^1_1$-conservative over $\RCA_0 + \BSig_2$.
\end{maintheorem}

This improves a theorem by Le Houérou and Patey~\cite{houerou2024pi} who proved that $\RCA_0 + \OVW^0$ is $\forall \Pi^0_4$-conservative over $\RCA_0 + \BSig_2$, where a formula is $\forall \Pi^0_4$ iff it is of the form $\forall X \varphi(X)$ for a $\Pi^0_4$-formula~$\varphi$. Note that this conservation is sharp, as $\OVW^0$ restricted to unary alphabet is the infinite pigeonhole principle ($\RT^1$), which is equivalent to $\BSig_2$ over~$\RCA_0$ (see Hirst~\cite[Theorem 6.4]{hirst1987combinatorics}). 

\subsection{Tree theorem for singletons}

One can see the Ordered Variable Word theorem as a tree partition theorem as follows: Given a finite alphabet~$A$ and an ordered $\omega$-variable sequence~$W$ over~$A$, letting $T_W = \{ W[u] : u \in A^{<\omega} \}$, $(A^{<\omega}, \preceq)$ and $(T_W, \preceq)$ are tree-isomorphic, with witness map $u \mapsto W[u]$. The Ordered Variable Word theorem is therefore a generalization of the tree theorem of Chubb, Hirst, and McNicholl~\cite{chubb2009reverse}. The tree theorem for singletons ($\mathsf{TT}^1$) states that for every $\ell \geq 1$ and every coloring $f : 2^{<\omega} \to \ell$, there is an infinite $f$-monochromatic set $S \subseteq 2^{<\omega}$ such that $(S, \preceq)$ and $(2^{<\omega}, \preceq)$ are tree-isomorphic.

The first-order part of the tree theorem for singletons was a long-standing open question. Corduan, Groszek and Mileti~\cite{corduan2010reverse} proved that $\mathsf{TT}^1$ is not provable over $\RCA_0 + \BSig_2$. Chong, Wang and Yang~\cite{chong2023conservation} proved that $\RCA_0 + \mathsf{TT}^1$ is $\forall \Pi^0_3$-conservative over~$\RCA_0 + \BSig_2$ and Chong, Li, Wang and Yang~\cite{chong2020strength} proved that $\RCA_0 + \mathsf{TT}^1 + \PSig_1$ is $\Pi^1_1$-conservative over $\RCA_0 + \BSig_2 + \PSig_1$, where $\PSig_1$ is equivalent to the principle of the totality of Ackermann's function (see Kreuzer and Yokoyama~\cite{kreuzer2016principles}).

\begin{corollary}
$\RCA_0 + \mathsf{TT}^1$ is $\Pi^1_1$-conservative over~$\RCA_0 + \BSig_2$.
\end{corollary}
\begin{proof}
By Main Theorem~\ref{thm:main1}, it suffices to prove that $\RCA_0 \vdash \OVW^0 \to \TT^1$.
Let $f : 2^{<\omega} \to k$ be an instance of~$\mathsf{TT}^1$. By $\OVW^0$, there is an ordered $\omega$-variable sequence~$W$ over $A = \{0,1\}$ such that the set $\{W[u] : u \in A^{<\omega} \}$ is $f$-monochromatic. Note that $\{W[u] : u \in A^{<\omega} \}$ is a $\mathsf{TT}^1$-solution to~$f$.
\end{proof}

Together with the recent proof by Hwang~\cite{hwang2026conservation} that Ramsey's theorem for pairs and two colors ($\RT^2_2$) is a $\Pi^1_1$-conservative extension of~$\RCA_0 + \BSig_2$, this can be considered as one more argument towards the robustness of the levels of the inductive hierarchy for capturing the first-order parts of natural second-order statements.

Note that Dzhafarov, Solomon and Yokoyama~\cite{dzhafarov2024firstorder} defined a semantic notion of first-order part of a Weihrauch problem~$\Psf$, corresponding to the strongest Weihrauch problem whose solutions range over~$\NN$, and which is reducible to~$\Psf$. Dzhafarov, Solomon and Valenti~\cite{dzhafarov2026tree} proved that the Weihrauch first-order part of $\mathsf{TT}^1$ and $\RT^1$ coincide.

\section{Prerequisites}\label[section]{sect:prerequisites}

Before proving our main theorem, we recall some fundamental concepts of weak arithmetics which will be used in the proof, and fix some notation for variable words. 

\subsection{Weak theories}

Let $\PAm$ be the axioms of a discrete ordered commutative semiring and $\Exp$ be the statement of the totality of the exponential.

\begin{definition}[Induction scheme]
Given a family of formulas $\Gamma$, let $\mathsf{I}\Gamma$ be the following scheme for every $\Gamma$-formula $\varphi$:
$$
[\varphi(0) \wedge \forall x(\varphi(x) \to \varphi(x+1))] \to \forall y \varphi(y).
$$
\end{definition}

We are in particular interested in the schemata $\ISig_n$ and $\IPi_n$ corresponding to $\Sigma^0_n$ and $\Pi^0_n$-formulas, respectively.
The induction scheme can be extended to $\Delta^0_n$-predicates as follows:

\begin{definition}
Given $n \geq 1$, let $\IDelta_n$ be the following scheme for every $\Sigma^0_n$-formula $\varphi$ and every $\Pi^0_n$-formula $\psi$:
$$
[\forall x(\varphi(x) \leftrightarrow \psi(x)) \wedge \varphi(0) \wedge \forall x(\varphi(x) \to \varphi(x+1))] \to \forall y \varphi(y).
$$  
\end{definition}

The least induction scheme for $\Gamma$-formulas ($\mathsf{L}\Gamma$) informally says for every $\Gamma$-formula~$\varphi$, if $\varphi(x)$ holds for some~$x$, then there is a least such~$x$. The least induction scheme for $\Delta^0_n$-predicates ($\LDelta_n$) is defined accordingly.
The following collection (or bounding) principle plays an important role in weak arithmetic:

\begin{definition}[Collection scheme]
Given a family of formulas $\Gamma$, let $\mathsf{B}\Gamma$ be the following scheme for every $\Gamma$-formula $\varphi$ and every $a \in \NN$
$$
[\forall x < a\exists y \varphi(x, y)] \to \exists b \forall x < a\exists y < b\varphi(x, y)
$$
\end{definition}

The following table summarizes the hierarchy of bounding and induction principles over $\PAm + \IDelta_0 + \Exp$. See Paris and Kirby~\cite{parisNCollectionSchemasArithmetic1978} for a reference:

\begin{center}
\bgroup
\def\arraystretch{1.25}%
\begin{tabular}{|c|c|c|} \hline
Induction & Collection & Least principle \\ \hline
$\vdots$ & $\vdots$ & $\vdots$ \\ \hline
$\ISig_2 \equiv \IPi_2$ & & $\LPi_2 \equiv \LSig_2$ \\ \hline
$\IDelta_2$ & $\BSig_2 \equiv \BPi_1$ & $\LDelta_2$ \\ \hline
$\ISig_1 \equiv \IPi_1$ & & $\LPi_1 \equiv \LSig_1$ \\ \hline
$\IDelta_1$ & $\BSig_1 \equiv \BPi_0$ & $\LDelta_1$ \\ \hline
\end{tabular}
\egroup
\end{center}

Thanks to the correspondence between computability and definability, the computable sets are exactly the ones which are definable by a $\Delta_1$-predicate. The following scheme is therefore known as \emph{recursive comprehension scheme}:

\begin{definition}
Let $\GammaCA{\Delta^0_1}$ be the following scheme for every $\Sigma^0_1$-formula~$\varphi$ and $\Pi^0_1$-formula~$\psi$:
$$
\forall x(\varphi(x) \leftrightarrow \psi(x)) \to \exists X \forall x(\varphi(x) \leftrightarrow x \in X).
$$  
\end{definition}

The base theory of reverse mathematics is $\RCA_0$, standing for Recursive Comprehension Axiom, defined as $\PAm + \GammaCA{\Delta^0_1} + \ISig_1$. More recently, Simpson and Smith~\cite{simpson1986factorization} defined a weaker theory, $\RCA_0^*$, which consists of $\PAm + \GammaCA{\Delta^0_1} + \IDelta_1 + \Exp$. In particular, $\RCA_0^* \vdash \BSig_1$.
The $\RCA_0$-provably total computable functions are exactly the primitive recursive ones, while the $\RCA_0^*$-provably total computable functions are the elementary recursive ones.

\subsection{Models of arithmetic}

Structures in second-order arithmetic are of the form
$$
\M = (M, S, 0, 1, +, \cdot, <)
$$
where $M$ is the first-order part, $S \subseteq \P(M)$ is the second-order part, $0, 1 \in M$, $+, \cdot : M^2 \to M$ are binary operators and $< \subseteq M^2$ is a binary relation. We shall abbreviate it with $(M, S)$. Given a structure $\M = (M, S)$, we will not assume in general that a set $X \subseteq M$ belongs to~$S$.

Fix a primitive recursive G\"odel numbering $\langle \cdot \rangle$ of mathematical objects. A mathematical object is \emph{$M$-coded} if its G\"odel number belongs to~$M$.

\begin{definition}
A set $X \subseteq M$ is \emph{$M$-regular} (or \emph{piecewise $M$-coded}) if every initial segment is $M$-coded.
\end{definition}

Given a structure $\M = (M, S) \models \PAm + \IDelta_0 + \Exp$, 
the $\ISig_n$ (resp.\ $\IPi_n$, $\IDelta_n$) scheme is equivalent to the statement \qt{Every $\Sigma^0_n$-definable (resp.\ $\Pi^0_n$-definable, $\Delta^0_n$-definable) set is $M$-regular} (see H\'ajek and Pudl\'ak~\cite[Chapter 2]{hajek1998metamathematics}).

\begin{definition}
Let $\M_0 = (M_0, S_0)$ and $\M_1 = (M_1, S_1)$ be two structures. Then $\M_1$ \emph{$\bbomega$-extends} $\M_0$ if $M_0 = M_1$ and $S_1 \supseteq S_0$.
\end{definition}

The notion of $\bbomega$-extension should not be confused with that of $\omega$-model. In particular, $M_0$ need not be the standard integers. Given a structure $\M = (M, S)$ and a set~$G \subseteq M$, we denote by $\M[G]$ the $\bbomega$-extension of~$\M$ whose second-order part consists of all $\Delta_1(S, G)$-definable sets.

% \begin{proposition}[Folklore]
% Let $T_0$ and $T_1$ be two second-order theories.
% If every countable model of $T_0$ $\bbomega$-extends to $T_1$,
% then $T_1$ is $\Pi^1_1$-conservative over~$T_0$.
% \end{proposition}
% \begin{proof}
% Suppose that $T_0 \nvdash \forall X\varphi(X)$ for an arithmetic formula $\varphi$. By the completeness theorem for Henkin models and the downward Löwenheim–Skolem theorem, there is a countable model $\M_0 = (M, S_0) \models T_0 + \exists X \neg \varphi(X)$. Let $\M_1 = (M, S_1) \models T_1$ be an $\bbomega$-extension of $\M_0$. In particular, $X \in S_1$ and since $\varphi$ is arithmetic, $\M_1 \models \neg \varphi(X)$, so $T_1 \nvdash \forall X \varphi(X)$.
% \end{proof}

A structure $\M = (M, S)$ is \emph{topped} by a set~$Y \in S$ if $S$ is the collection of all $\Delta_1(Y)$-definable sets.

For a structure $\M = (M, S)$, a set $X \subseteq M$ is \emph{$M$-bounded} if there is some~$b \in M$ such that for every~$x \in X$, $x < b$. Otherwise, $X$ is \emph{$M$-unbounded}. In models of~$\RCA_0^*$, some sets might be unbounded, but have finite cardinality. The statement \qt{Every unbounded set is in bijection with~$\NN$} is actually equivalent to $\Sigma^0_1$-induction. We can therefore use indistinctly \emph{infinite} and \emph{unbounded} in models of~$\RCA_0$.

\subsection{Formal computability theory}

In what follows, fix a structure $\M = (M, S) \models \RCA_0^*$.

\begin{definition}
Given $X, Y \in S$, we write $X \leq_T Y$ if $X$ is $\Delta_1(Y)$-definable.
\end{definition}

Turing reducibility is robust in models of $\RCA_0^*$. In particular, it is transitive.
Fix a universal $\Sigma^0_1$-formula $\varphi(X, e)$.

\begin{definition}
Given $X \in S$, let $X' = \{ e : \varphi(X, e) \}$.
\end{definition}

If $X \leq_T Y$ with $Y' \in S$, then $X' \in S$ and $X' \leq_T Y'$. Moreover $Y \leq_T Y'$ and $Y' \not \leq_T Y$.

\begin{definition}
Given two sets $X, Y \in S$, let $$X \oplus Y = \{ 2x : x \in X \} \sqcup \{ 2x+1 : x \in Y \}$$    
\end{definition}

We have $X \leq_T X \oplus Y$ and $Y \leq_T X \oplus Y$. Moreover, $X \oplus Y$ is $\Delta_1(X, Y)$-definable.
See Chong and Yang~\cite{chongJumpSigma_ncut2007} for a more detailed exposition of computability theory in models of weak arithmetic.

\begin{lemma}[Folklore]\label[lemma]{lem:jump-bsig2-bsig1}
Let $\M = (M, S)$ be a countable structure topped by a set~$Y$.
Suppose that $Y' \leq_T P$ for some set $P \subseteq M$ such that $\M[P] \models \RCA_0^*$.
Then $\M \models \RCA_0 + \BSig_2$.
\end{lemma}
\begin{proof}
Since $\M[P] \models \PAm + \Exp$, so does~$\M$.
By definition of a topped model, $\M \models \GammaCA{\Delta^0_1}$.
Any $\Sigma_1(\M)$-formula can be translated into a $\Delta^0_1(\M[P])$-formula, 
so any instance of $\ISig_1$ in~$\M$ becomes an instance of~$\IDelta_0$ in~$\M[P]$. It follows that $\M \models \ISig_1$, so $\M \models \RCA_0$.
Any $\Sigma^0_2(\M)$-formula can be translated into a $\Sigma^0_1(\M[P])$-formula,
so any instance of $\BSig_2$ in $\M$ becomes an instance of $\BSig_1$ in~$\M[P]$.
Thus, $\M \models \BSig_2$.
\end{proof}

\subsection{Weak K\"onig's lemma}

A \emph{tree} is a subset $T \subseteq \NN^{<\bbomega}$ which is closed under prefix. A \emph{binary tree} is a tree $T \subseteq 2^{<\bbomega}$. A \emph{path} through a tree~$T$ is a sequence $P \in \NN^\NN$ such that every initial segment belongs to~$T$. We write $[T]$ for the (possibly empty) collection of all paths through~$T$.

\begin{statement}[Weak K\"onig's lemma]
$\WKL$: Every infinite binary tree admits an infinite path.
\end{statement}

We shall write $\WKL_0^*$ and $\WKL_0$ for the theories~$\RCA_0^* + \WKL$ and $\RCA_0 + \WKL$, respectively.  
Simpson and Smith~\cite{simpson1986factorization} proved the following theorem:

\begin{proposition}[Simpson and Smith~\cite{simpson1986factorization}]\label[proposition]{prop:rca-wkl-star}
Let $\M = (M, S) \models \RCA_0^*$, and $T$ be an infinite binary tree in~$S$. 
There is an infinite path~$P$ such that $\M[P] \models \RCA_0^*$.
\end{proposition}

We shall also need the following formalized low basis theorem. Note that since $\M \models \BSig_2$, then $\M[Y'] \models \RCA_0^*$, so the notion of Turing reducibility is still robust.

\begin{proposition}[H\'ajek~\cite{hajek1993interpretability}]\label[proposition]{prop:wkl-formalized-low}
Consider a countable model $\M = (M,S) \models \RCA_0 + \BSig_2$ topped by a set $Y \in S$.
Then for every infinite tree $T \in S$, there exists $G \subseteq M$ such that 
$G$ is a path in $T$ and $Y' \geq_T (G \oplus Y)'$.
In particular, $\M[G] \models \RCA_0 + \BSig_2$
\end{proposition}

Recall that given an index of a co-$X$-c.e.\ tree~$T$, one can computably obtain an index of a primitive $X$-recursive infinite tree~$S$ such that if $[T] \neq \emptyset$, then $[S] = [T]$.
Let $(T^X_n)_{n \in \NN}$ be an effective listing of all $X$-primitive recursive infinite binary trees.

\begin{definition}
A set $Q = \bigoplus_n Z_n$ \emph{codes a Scott ideal} if for every $i, j \in \NN$,
$$Z_{\langle i, j\rangle} = [T_i^{Z_0 \oplus \cdots \oplus Z_j}]$$
\end{definition}

If~$Q = \bigoplus_n Z_n$ codes a Scott ideal, then $\I = \{ Z_n : n \in \NN \}$ is a Turing ideal whose closure properties are effective. Given any~$X \in \I$, a \emph{$Q$-index} of~$X$ is any~$n \in \NN$ such that $X = Z_n$. Given $e, i \in \NN$, one can compute some~$j \in \NN$ such that if $\Phi_e^{Z_i}$ is total, then it is equal to~$Z_j$. The basic operations such as $(i, j) \mapsto k$ such that $X_k = X_i \oplus X_j$ are also computable.

\begin{lemma}[Folklore]\label[lemma]{lem:wkl-scott-ideal}
Let $\M = (M, S) \models \RCA_0^*$ and $X \in S$.
There exists a $\Delta_1(X)$-definable infinite binary tree~$T \in S$ such that every $Q \in [T]$ codes a Scott ideal containing~$X$.
\end{lemma}

Weak K\"onig's lemma can actually prove the existence of paths through a larger family of finitely branching trees: the ones for which one can uniformly bound the maximum label at every level. A tree $T \subseteq \NN^{<\bbomega}$ is \emph{$g$-bounded} for some function $g : \NN \to \NN$ if for every $\sigma \in T$ and $x < |\sigma|$, $\sigma(x) < g(x)$.

\begin{lemma}[Folklore]\label[lemma]{lem:wkl-to-bounded}
$\WKL_0^*$ proves that for every infinite tree $T \subseteq \NN^{<\bbomega}$ which is $g$-bounded for some $g : \NN \to \NN$, there is an infinite path $P \in [T]$.
\end{lemma}
\begin{proof}[Proof sketch]
Assume without loss of generality that $g(x) = 2^{h(x)}$ for some function $h : \NN \to \NN$.
Encode every level~$x$ of $T$ into $h(x)$ many levels of a binary tree.
\end{proof}

% Every $\Sigma^0_1$ formula $\varphi(G)$ can be expressed in normal form as $\exists k \psi(G \uh_k)$, where $\psi(\sigma)$ is a monotonous $\Sigma^0_1$ formula, that is, if $\psi(\sigma)$ and $\sigma \preceq \tau$, then $\psi(\tau)$ holds. 
% All over this paper, we shall use the following standard compactness argument:

% \begin{lemma}[Folklore]\label[lemma]{lem:compactness-big2-wkl}
% $\WKL_0 + \BSig_2$ proves that for every $\Sigma^0_1$ formula~$\varphi(G) \equiv \exists k \psi(G \uh k)$ in normal form such that $\forall A \varphi(A)$ holds, there exists some $N_1 \in \NN$ such that for every $\sigma \in 2^{N_1}$, $\psi(\sigma)$ holds. Moreover, for every $\Delta^0_2$ predicate~$A$, $\varphi(A)$ holds.
% \end{lemma}
% \begin{proof}
% Suppose there is no such~$N_1$. Let~$T$ be the tree of all $\sigma \in 2^{<\bbomega}$ such that $\neg \psi(\sigma)$. By monotonicity of~$\psi$, $T$ is a tree, and by assumption, it is infinite. By $\WKL_0$, there is some~$A \in [T]$. Then for every~$k$, $\neg \psi(A \uh k)$, so $\neg \varphi(A)$, contradicting our assumption.
% Let~$N_1$ witness the statement of the lemma and let $A$ be a $\Delta^0_2$ predicate. By $\BSig_2$, $A$ is regular, so $A \uh_{N_1}$ is coded. By choice of~$N_1$, $\psi(A \uh_{N_1})$ holds, so $\exists k \psi(A \uh_k)$ holds, hence $\varphi(A)$ holds.
% \end{proof}

\subsection{Variable words}

We shall adopt the same notation conventions as~\cite{liu2026reverse} for variable words, that we now recall. Those definitions have to be understood in the formal setting. In particular, $\bbomega$ will be interpreted by $M$ rather than~$\omega$ in a structure $\M = (M, S)$.

A finite or infinite \emph{word} over a finite alphabet~$A$ is a map $w : \alpha \to A$ for some~$\alpha \leq \bbomega$.
We write $|w|$ for its \emph{length}~$\alpha$. Given $\beta \leq \bbomega$, an \emph{ordered $\beta$-variable word} over~$A$ is a word $w$ over the alphabet $A \sqcup \{ x_i : i < \beta \}$ where each $x_i$ appears at least once, and the last occurrence of $x_i$ appears before the first occurrence of $x_{i+1}$. We write $A^{\alpha, \beta}_{<}$ and $A^{<\alpha, \beta}_{<}$ for the sets of all ordered $\beta$-variable word of length~$\alpha$, and smaller than~$\alpha$, respectively. Given $\gamma \leq \beta \leq \bbomega$, a $\beta$-variable word~$w$ over~$A$ and a $\gamma$-variable word~$u$ of length~$|u| \leq \beta$, we write $w[u]$ for the $\gamma$-variable word over~$A$ where each occurrence of $x_j$ is replaced with $u(j)$, and cut before the first occurrence of $x_{|u|}$. If $|u| = \beta$, then $|w[u]| = |w|$.

Given ordered variable words~$u$ and $v$, when doing the concatenation $u \cdot v$, we will always assume that the variable kinds of~$u$ and $v$ are from disjoint spaces, so that there is no variable kind clash.
Given $\gamma \leq \beta \leq \bbomega$ and an ordered $\beta$-variable word~$w$, we let $\OSub^\gamma_{A}(w)$ be the set of all ordered $\gamma$-variable words generated by~$w$, that is,
$$
\OSub^\gamma_{A}(w) =  \{ w[u] : u \in A^{\leq\beta, \gamma}_{<} \}
$$
Most of the time, the alphabet~$A$ will be clear from context, and we shall simply write $\OSub^\gamma(w)$ for $\OSub^\gamma_A(w)$.
Let
\begin{align}
 \OSub(w)&:= \bigcup\limits_{0\leq \gamma\leq \beta}\OSub^\gamma(w).
 \end{align}
Note that $\OSub(w)$ includes in particular $\OSub^0(w)$, that is, the set of words generated by~$w$. We write $\WSub(w)$ for $\OSub^0(w)$.
Last, we define
\begin{align}\nonumber
\OSub^=(w) &:= \big\{ v\in \OSub(w):|v| = |w|  \big\}\\ \nonumber
\OSub^*(w) &:=
\big\{
v\in\OSub(w): |v|<|w|
\big\}.
\end{align}
The sets $\OSub^{\gamma,=}(w)$, $\OSub^{\gamma,\star}(w)$, $\WSub^=(w)$ and $\WSub^*(w)$ are defined accordingly.

% \subsection{Infinitely often OVW}

% The notion of forcing used by Liu and Patey~\cite{liu2026reverse} to build solutions to $\OVW^0$ actually constructs a weaker object, which is monochromatic at only infinitely many levels.

% \begin{definition}
% Let $f : A^{<\bbomega} \to \ell$ be a coloring for a finite alphabet~$A$ and some~$\ell \geq 1$.
% An ordered $\bbomega$-variable word~$W$ is \emph{$f$-infinitely-often-monochromatic} if for infinitely many $n\in\NN$, $W[A^n] = \{ W[u] : u \in A^n \}$ is $f$-monochromatic.
% \end{definition}

% \ludovic{TODO, maybe ignore the concept and integrate it into the proof.}

\section{Conservation theorem}\label[section]{sect:conservation}

The goal of this section is to prove the following one-step extension theorem.

\begin{proposition}\label[proposition]{prop:ovw-step}
Consider a countable model $\M = (M,S) \models \RCA_0 + \BSig_2$ topped by a set $Y \in S$. Fix an $M$-bounded alphabet~$A$, some $\ell \in M$ and some set~$P \gg Y'$ such that $\M[P] \models \RCA_0^*$. For every coloring $f : A^{<M} \to \ell$ in~$S$, there is an ordered $M$-variable sequence~$W$ over~$A$ such that  
\begin{enumerate}
    \item $\WSub^*(W)$ is $f$-monochromatic ;
    \item $P \geq_T (W \oplus Y)'$.
\end{enumerate}
\end{proposition}

Before proving \Cref{prop:ovw-step}, let us deduce the $\Pi^1_1$-conservation theorem from it.
First, by Simpson and Smith~\cite{simpson1986factorization}, one can always find a set~$P$ satisfying the hypothesis of \Cref{prop:ovw-step}. We therefore deduce the following corollary:

\begin{corollary}\label[corollary]{cor:ovw-step-improved}
Consider a countable model $\M = (M,S) \models \RCA_0 + \BSig_2$ topped by a set $Y \in S$. Fix an $M$-bounded alphabet~$A$, some $\ell \in M$, and a coloring $f : A^{<M} \to \ell$ in~$S$.
Then there is an ordered $M$-variable sequence~$W$ over~$A$ such that 
\begin{enumerate}
    \item[(1)] $\WSub^*(W)$ is $f$-monochromatic ;
    \item[(2)] $\M[W] \models \RCA_0 + \BSig_2$.
\end{enumerate}
\end{corollary}
\begin{proof}
Let $\M' = (M, \Def{\Delta_2}(Y,M))$.
Since $\M \models \RCA_0 + \BSig_2$, then $\M' \models \RCA_0^*$.
Let $T \subseteq 2^{<M}$ be a $\Delta_2(Y)$-definable infinite tree such that for every path~$P \in [T]$, $P \gg Y'$. In particular, $T \in \M'$, so by \Cref{prop:rca-wkl-star}, there is a set~$P \subseteq M$ with $P \gg Y'$ such that $\M'[P] \models \RCA_0^*$.
By \Cref{prop:ovw-step}, there is an ordered $M$-variable sequence~$W$ over~$A$ such that  $\WSub^*(W)$ is $f$-monochromatic and $P \geq_T (W \oplus Y)'$. Since $\M'[P] \models \RCA_0^*$, then by \Cref{lem:jump-bsig2-bsig1}, $\M[W] \models \RCA_0 + \BSig_2$.
\end{proof}

Iterating \Cref{cor:ovw-step-improved}, one obtains the following $\bbomega$-extension theorem:

\begin{proposition}\label[proposition]{prop:ovw-omega-extension}
For every countable topped model $\M = (M,S) \models \RCA_0 + \BSig_2$,
there is an $\bbomega$-extension~$\M_\infty = (M, S_\infty) \models \RCA_0 + \OVW^0$.
\end{proposition}
\begin{proof}
Using \Cref{cor:ovw-step-improved}, construct a countable sequence of structures $\M = \M_0, \M_1, \M_2, \dots$ such that for every $n \in \omega$,
\begin{itemize}
    \item[(1)] $\M_{n+1} = (M, S_{n+1})$ is a topped model $\bbomega$-extending $\M_n$;
    \item[(2)] $\M_{n+1} \models \RCA_0 + \BSig_2$;
    \item[(3)] For every $M$-bounded alphabet~$A$, every $\ell \in M$ and every coloring $f : A^{<M} \to \ell$ in~$S_n$, there is some~$k \in \omega$ and some ordered $M$-variable sequence~$W$ over~$A$ such that $\WSub^*(W)$ is $f$-monochromatic and $W \in S_k$.
\end{itemize}
Then $\M_\infty = (M, \bigcup_n S_n)$ is an $\bbomega$-extension of~$\M$ such that $\M_\infty \models \RCA_0 + \BSig_2 + \OVW^0$.
\end{proof}

We can now prove the main theorem:

\begin{repmaintheorem}{thm:main1}
$\RCA_0 + \OVW^0$ is $\Pi^1_1$-conservative over $\RCA_0 + \BSig_2$.
\end{repmaintheorem}
\begin{proof}
Suppose that $\RCA_0 + \BSig_2 \nvdash \forall X \varphi(X)$ for some arithmetic formula~$\varphi$. By the completeness theorem for Henkin models and the downward Löwenheim–Sko\-lem theorem, there is a countable model $\M_0 = (M, S_0) \models \RCA_0 + \BSig_2 + \exists X \neg \varphi(X)$. Let $X \in S_0$ be such that $\M_0 \models \neg \varphi(X)$ and $\M_1 = (M, \Def{\Delta^0_1}(X, M))$. Then $\M_1$ is a countable model $\M_1 \models \RCA_0 \wedge \BSig_2 \wedge \neg \varphi(X)$ which is topped by~$X$.
By \Cref{prop:ovw-omega-extension}, there is an $\bbomega$-extension $\M_2 = (M, S_2)$ of~$\M_1$
such that $\M_2 \models \RCA_0 + \OVW^0$. Since $S_2 \supseteq S_1$, then $X \in S_2$, and since $\varphi$ is arithmetic, $\M_2 \models \neg \varphi(X)$, so $\RCA_0 + \OVW^0 \nvdash \forall X \varphi(X)$.
\end{proof}

We now turn to the proof of \Cref{prop:ovw-step}.
In what follows, fix a countable model $\M = (M, S)$ topped by a set $Y \in S$, fix an $M$-bounded alphabet $A$, some number of colors $\ell \in M$, a set $P \gg Y'$ and a coloring $f : A^{<M} \to \ell$ in~$S$ as in the hypothesis of \Cref{prop:ovw-step}.

As in~\cite{liu2026reverse}, we will design a notion of forcing such that any sufficiently generic filter produces an infinite, finitely branching tree~$T$. Every infinite path through~$T$ is an ordered $M$-variable word~$W$ over~$A$ which is $f$-monochromatic \emph{at infinitely many levels},
that is, such that for infinitely many $n\in M$, $W[A^n] = \{ W[u] : u \in A^n \}$ is $f$-monochromatic. Then, an application of the infinite pigeonhole principle (which holds over~$\RCA_0 + \BSig_2$) yields an $\OVW^0$-solution.

Then, we will construct a $Y'$-computable 1-generic decreasing sequence of forcing conditions, so that every path through the generic tree is GL${}_1$. The construction being internal, the main difficulty is to simultaneously decide the jump on initial segments of the generic set, using the technique of Shore blocking.

The notion of forcing being Mathias-like, with infinite reservoirs, we first fix a Scott ideal to which belong the reservoirs. By \Cref{prop:wkl-formalized-low} and \Cref{lem:wkl-scott-ideal}, there is a set~$Q = \bigoplus_{a \in M} X_a$ such that $(Q \oplus Y)' \leq_T Y'$ and $Q$ codes the Scott ideal $\{ X_a : a \in M \}$.

\subsection{Notion of forcing}\label[section]{sec:forcing-notion}

The following notion of $\PP$-branch is a coded representation of the corresponding notion in~\cite{liu2026reverse}. 

\begin{definition}\label[definition]{caOVW-def-precondition}
A \emph{$\PP$-branch} is a pair $(v,a)$ where $v$ is an $M$-coded ordered variable word and
$X_a$ an ordered $M$-variable word.
\end{definition}

Its \emph{denotation} $[v, a]$ is defined as 
\begin{align}\label{caOVW-eq14}
[v, a] = \big\{
G\in \OSub(v\cdot X_a) :\ & v\preceq G
\big\}
\end{align}
Note that $G$ might be of any length between $|v|$ and $M$.
We also implicitly require that the variable kinds in $G\uhr_{|v|}^{|G|-1}$ and $v$ are disjoint.

\begin{definition}[$\PP$-branch extension]\label[definition]{caOVW-def-preconditionextension}
A $\PP$-branch\ $(v, b)$ \emph{extends}
a $\PP$-branch\ $(u,a)$
(written $(v, b)\leq (u,a)$)
iff
$v \succeq u$
and $v\cdot X_b \in \OSub^M(u\cdot X_a)$.
% Or equivalently, for some ordered $M$-variable word
%  $U\in \OSub^M(X_a)$ and some $t\in\omega$,
%  $v = u\cdot (U\uh t)$ and $X_b = U\uhr_t^\infty$
%  where $t$ must be in between variables of $U$.
\end{definition}

As noted in~\cite{liu2026reverse}, if $(v, b)\leq (u,a)$, we do not have in general $v \in \OSub^*(u\cdot X_a)$. One easily checks that $\PP$-branch extension is  transitive, and that if $(v, b) \leq (u, a)$, then $[v, b] \subseteq [u, a]$. Note that the extension relation is $\Pi_1(Q)$: $W' \in \OSub^M(W)$ is $\Pi_1(W \oplus W')$.

\begin{definition}[Forcing condition]\label[definition]{caOVW-def-condition}
 A \emph{$\PP$-precondition}
is a pair $(I,a)$ where
\begin{itemize}
\item $I$ is an $M$-coded set of ordered variable word
that is length-homogeneous. i.e., For some $t\in M$ (denoted as $t(I)$)
 $|v| = t$ for all $v\in I$;
\item For every $v\in I$, $(v,a)$
is a $\PP$-branch.
\end{itemize}
A $\PP$-precondition $(I, a)$ is \emph{$f$-matching} if for every word $w\in \WSub^*(X_a)$,
there exists $v\in I$
such that $\WSub^=(v\cdot w)$ is $f$-monochromatic.
A \emph{$\PP$-condition} is an $f$-matching $\PP$-precondition.
\end{definition}

The \emph{denotation} of a $\PP$-precondition $(I, a)$ is $[I,a] = \cup_{v\in I} [v,a]$.
As explained in~\cite{liu2026reverse}, the denotation of a $\PP$-branch or of a $\PP$-precondition may not contain any ordered $M$-variable word which is $f$-monochromatic at infinitely many levels, but the $f$-matching property ensures that a $\PP$-condition contains sufficiently many stems for it to be the case (see \cite[Lemma 3.12]{liu2026reverse}).

\begin{definition}[$\PP$-precondition extension]
A $\PP$-precondition $(J,b)$ \emph{extends} $\PP$-precondition
$(I,a)$ iff for every $v \in J$,
there exists an $u\in I$ such that
$(v,b)\leq (u,a)$.
\end{definition}

Transitivity of $\PP$-precondition extension is inherited from that of $\PP$-branch extension,
and if $(J, b) \leq (I, a)$, then $[J, b] \subseteq [I, a]$. Last, it is useful to give an explicit syntactical definition of forcing relation for $\Sigma^0_1$ and $\Pi^0_1$-formulas that satisfies \qt{forcing implies truth} property for every (i.e. even non-generic) filters. 

Every $\Sigma^0_1$ formula $\varphi(G)$ can be expressed in normal form as $\exists k \psi(G \uh_k)$, where $\psi(\sigma)$ is a monotonous $\Delta^0_0$ formula, that is, if $\psi(\sigma)$ and $\sigma \preceq \tau$, then $\psi(\tau)$ holds. We shall always assume that the $\Sigma^0_1$-formulas are in normal form.

\begin{definition}[Forcing relation]\label[definition]{caOVW-def-forcing}
Let $(v,a)$ be a $\PP$-branch and $\varphi(G):= \exists n\psi(G\uhr n)$ be a $\Sigma^0_1$-formula. Let
\begin{itemize}
    \item[(1)] $(v,a)\Vdash \varphi(G)$
    iff $\psi(v\uhr n)$ holds for some $n$;
    \item[(2)] $(v,a)\Vdash \neg\varphi(G)$ iff
    for every finite $\h v\in [v,a]$, every $n\in M$, $\neg\psi(\h v\uhr n)$ holds.
\end{itemize}
\end{definition}

We say that a $\PP$-branch $(v, a)$ \emph{decides} a $\Sigma^0_1$-formula $\varphi(G)$ if either $(v,a)\Vdash\varphi(G)$ or $(v,a)\Vdash\neg\varphi(G)$.
Accordingly, a $\PP$-precondition $(I,a)$ \emph{decides} a $\Sigma^0_1$-formula if each of its $\PP$-branches decides it.

Assuming the $\Sigma^0_1$-formula $\varphi(G)$ is in Kleene's normal form, if $(u, a) \Vdash \varphi(G)$ and $(v, b) \leq (u, a)$, then $(v, b) \Vdash \varphi(G)$. Moreover, $\varphi(G)$ holds for every $G \in [u, a]$.
The same holds for $\neg \varphi(G)$.

\subsection{Graham-Rothschild tree}\label[section]{sect:graham-rothschild-tree}

The notion of Graham-Rothschild tree is a fundamental notion in the combinatorics of the Ordered Variable Word theorem. Intuitively, it is an infinite, finitely branching tree whose infinite paths are ordered $\bbomega$-variable words and such that every level contains \qt{sufficiently many} stems so that the corresponding $\PP$-precondition is a $\PP$-condition.

It is based on infinite iterations of the Hales-Jewett theorem~\cite{hales1963regularity}. This finitary theorem is the bootstrap of a large part of Ramsey theory. The original proof used double induction and yielded a bound $HJ$ growing as fast as the Ackermann function. Shelah~\cite{shelah1988primitive} gave a new proof yielding primitive recursive bounds:

\begin{theorem}[Hales-Jewett, $\RCA_0$]\label[theorem]{thm:hales-jewett}
There exists a primitive recursive function~$HJ(k, \ell)$ such that for every $n \geq HJ(k, \ell)$, every alphabet~$A$ of size~$k$ and every coloring $f : A^n \to \ell$, there exists some 1-variable word $v$ over~$A$ of length~$n$ such that $\WSub^{=}(v) = \{ v[a] : a \in A \}$ is $f$-homogeneous.
\end{theorem}

For simplicity, given a finite alphabet~$A$ and some~$\ell \geq 1$, we might simply write $HJ(A,\ell)$ for $HJ(\card A, \ell)$.
Let $(\ell_n)_{n \in \NN}$ be inductively defined by
$$\ell_0 = \ell \mbox{ and } \ell_{n+1} = \ell_n \times \card A^{HJ(A, \ell_n),1}_<$$
Here, $\card A^{HJ(A, \ell_n),1}_<$ is the number of ordered 1-variable words of length $HJ(A, \ell_n)$ over~$A$. Then, let $(t_n)_{n \in \NN}$ be inductively defined by
$$t_0 = 0 \mbox{ and } t_{n+1} = t_n + HJ(A, \ell_n)$$

The \emph{Graham-Rothschild tree for $\ell$-colorings over alphabet~$A$} is the tree $\T_{A,\ell}$ whose level~$n$ consists of all ordered $n$-variable words~$w$ of length~$t_n$, such that for every~$s < n$, the variable~$x_s$ appears only within the positions $[t_s, t_{s+1})$. This tree is partially ordered by the prefix relation. The following property was proven by Liu and Patey~\cite{liu2026reverse} over~$\RCA_0$.

\begin{lemma}[Liu and Patey~\cite{liu2026reverse}, $\RCA_0$]\label[lemma]{lem:graham-rothschild-tree}
Fix~$A$ and $\ell \geq 1$.
Let $f : A^{<\bbomega} \to \ell$ be a coloring.
For every~$u \in A^{<\bbomega}$ and $n \in \NN$, there is some ordered $n$-variable word~$w$ at level~$n$ in~$\T_{A,\ell}$ such that $\WSub^{=}(w \cdot u)$ is $f$-monochromatic.
\end{lemma}

For the purpose of Shore blocking, we shall be interested in products of GR-trees:
Given $d \in M$, let $\T^d_{A,\ell}$ be the tree such that the nodes at level~$n$ are exactly the $d$-tuples of nodes at level~$n$ in $\T^d_{A,\ell}$. Any infinite path through $\T^d_{A,\ell}$ is a $d$-tuple $\vec{V} = (V_0, \dots, V_{d-1})$ of ordered $M$-variable words over~$A$. Given an ordered variable word~$u$, we write $\vec{V}[u]$ for 
$$
V_0[V_1[\cdots V_{d-1}[u] \cdots ]] = V_0[V_1] \cdots [V_{d-1}][u]
$$

\subsection{Deciding formulas simultaneously}

We now turn to the heart of the formalized weakly low basis theorem, which is proving the density of the set of conditions deciding a bounded quantity of formulas simultaneously. Iteratively applying the forcing question of Liu and Patey~\cite{liu2026reverse} would require too much induction strength. One therefore has to develop some new combinatorics.
The following forcing question is essentially \cite[Definition 3.16]{liu2026reverse} reformulated to emphasize the fact that it is a question within a branch rather than a precondition:

Recall that $\T_{A, k}$ is the Graham-Rotshchild tree
for $k$-colorings over alphabet~$A$ (see \Cref{sect:graham-rothschild-tree}).

\begin{definition}[Forcing question]\label[definition]{caOVW-def-forcingquestion}
Let $(v,a)$ be a $\PP$-branch and $k \in M$.
We write $(v, a)\qvdash_k\varphi(G)$
iff there exists a level~$n \in M$ such that for every $u$ at level~$n$ in~$\T_{A,k}$,
 there exists some $\hat u\in \OSub^=(u)$ such that
$\varphi(v\cdot X_a[\hat u])$ holds.
\end{definition}

The following $\Sigma^0_1$-extension lemma is an immediate adaptation of \cite[Lemma 3.19]{liu2026reverse}. The main difference is that one allows the formulas to depend on the stem~$v$. Note that this difference is only presentational since one can define $\varphi(G) := \varphi_{G \uh t}(G)$ where~$t$ is the common length of all elements in~$J$.

\begin{lemma}[$\Sigma^0_1$-extension]\label[lemma]{lem:sigma01-extension}
Let $(I, a)$ be a $\PP$-condition, $J \subseteq I$ and $(\varphi_v)_{v \in J}$ be uniformly $\Sigma_1(Y)$-formulas such that $(v, a) \qvdash_{\ell \cdot \card I} \varphi_v(G)$ for every~$v \in J$. Then there is a $\PP$-condition $(K, b) \leq (I, a)$ such that $(\h v, b) \Vdash \varphi_v(G)$ (i.e. $\varphi_v(\h v)$ holds) for every $v \in J$ and every $\hat v \in K$ extending~$v$.  

Moreover, $(K, b)$ can be computed from $(I, a)$, $J$ and $(\varphi_v)_{v \in J}$ using a $\Delta^0_1(Y')$ function. 
\end{lemma}
\begin{proof}
By $\BSig_1(X_a \oplus Y)$, there is a common level witness~$n$ for all $v \in J$. Note that the forcing questions are monotonic in their level witnesses, so they still hold for a simultaneous witness. Moreover, the map $(I, a, J) \mapsto n$ is $\Delta_1(Q \oplus Y)$, hence $\Delta_1(Y')$.

Let $m$ be the position of the first occurrence of~$x_{t_n}$ in $X_a$ (where $t_n$ is the length of all the nodes at level~$n$ in $\T_{A, \ell \cdot \card I}$). One can find $Q \oplus Y$-computably some~$M$-finite set~$K$ of all $v\cdot X_a[\hat u]$ such that either
\begin{itemize}
    \item[(i)] $v\in J$, $\h u\in \OSub^=(u)$ for some $u$ at level $n$ of $\T_{A,\ell \cdot \card I}$ and $\varphi_v(v \cdot X_a[\h u])$; or
    \item[(ii)] $v\in I\setminus J$ and $\hat u\in A^{t_n}$
\end{itemize}
Moreover, one can $Q \oplus Y$-computably find some~$b \in M$ such that $X_b = X_a \uhr_m^{\infty}$, with the variables renamed accordingly. The verification that $(K, b)$ is the desired extension is exactly the proof of \cite[Lemma 3.19]{liu2026reverse} and is formalizable over $\RCA_0$.
\end{proof}

The following $\Pi^0_1$-extension lemma constitutes the new combinatorial content of the article. For this lemma, we need to restrict ourselves to families of $\Sigma^0_1$-formulas $(\varphi_{v, i})_{v \in I, i < r}$ which satisfy some strong structural properties, that is, $\varphi_{v,i}(G) \to \varphi_{v,i+1}(G)$ for every~$G$, every~$v \in I$ and every $i < r-1$. We shall then generalize it to arbitrary families in \Cref{lem:simultaneous-decision} thanks to a counting trick.

\begin{lemma}[$\Pi^0_1$-extension]\label[lemma]{lem:pi01-extension}
Let $(I, a)$ be a $\PP$-condition, $r \in M$ and $(\varphi_{v, i})_{v \in I, i < r}$ be  uniformly $\Sigma_1(Y)$-formulas such that $\varphi_{v,i}(G) \to \varphi_{v,i+1}(G)$ for every~$G$, every~$v \in I$ and every $i < r-1$. There is some $\PP$-condition $(I, b) \leq (I, a)$ and a function $\beta : I \to \{0, \dots, r\}$ such that for every $v \in I$,
\begin{itemize}
    \item[(1)] $(v, b) \Vdash \neg \varphi_{v, \beta(v)-1}(G)$ if $\beta(v) > 0$;
    \item[(2)] $(v, b) \qvdash_{\ell \cdot \card I} \varphi_{v, \beta(v)}(G)$ if $\beta(v) < r$.
\end{itemize}
Moreover, $b$ and $\beta$ can be computed from $(I, a)$, $r$ and $(\varphi_{v, i})_{v \in I, i < r}$ by a $\Delta_1(Y')$ function.
\end{lemma}
\begin{proof}
%A \emph{profile} is a map $\beta : I \to \{0, \dots, r\}$. The \emph{weight} of $\beta$ is $|\beta| = \sum_{v \in I} \beta(v) \leq r \cdot \card I$.
Let $\T^d = \T^d_{A,\ell \cdot \card I}$ be the $d$-product of GR-trees as in the paragraph after \Cref{lem:graham-rothschild-tree}. Given $\vec{V} = (V_0, \dots, V_{d-1}) \in [\T^d]$, let 
$$X_a[\vec{V}] = X_a[V_0][V_1] \cdots [V_{d-1}].$$ 
Note that $X_a[\vec{V}][u] = X_a[\vec{V}[u]]$ for every finite~$u$. Let $\C^d_\beta$ be the $\Pi_1(X_a \oplus Y)$ class of all $\vec{V} \in [\T^d]$ such that
$$
(v, X_a[\vec{V}]) \Vdash \neg \varphi_{v, \beta(v)-1}(G) \text{ for every $v \in I$ with $\beta(v) > 0$.}
$$
Furthermore, one can find an index of an $X_a \oplus Y$-primitive recursive tree $T^d_\beta$ such that $[T^d_\beta] = \C^d_\beta$ computably in~$(a, d, \beta)$. Here, the forcing relation is natural generalization of \Cref{caOVW-def-forcing}(2). Let
$$
E(d, \beta) = \qt{\C^d_\beta \neq \emptyset}
$$
Equivalently, $E(d, \beta)$ holds iff $T^d_\beta$ has a node at every level, so $E(d, \beta)$ is a $\Pi_1(Q \oplus Y)$ property of $(a, d, \beta)$.
Since $\M \models \BSig_2$ and $(Q \oplus Y)' \leq_T Y'$, $\M[Q] \models \BSig_2$, so the set
$$
\langle E(d, \beta) : d \leq 1 + r \cdot \card I, \beta : I \to \{0, \dots, r\} \rangle
$$
is $M$-regular, and so is the function $g : 2 + r \cdot \card I \to 1 + r \cdot \card I$ defined by
$$
g(d) = \max \left\{ \sum_{v \in I} \beta(v) : E(d, \beta) \right\}
$$
Note that $E(d, \mathbf{0})$ holds for every~$d \leq 1 + r \cdot \card I$, where $\mathbf{0} : I \to \{0, \dots, r\}$ is the constant map $v \mapsto 0$. So the function $g$ is well-defined.
Moreover, if $E(d, \beta)$ holds, then so does $E(d+1, \beta)$. Indeed, appending any path~$V'$ of the Graham-Rothschild tree does not affect the $\Pi^0_1$ forcing relation. It follows that the function $g$ is non-decreasing. 

By the finite pigeonhole principle (which holds over~$\RCA_0$), there is some least~$d^*$ such that $g(d^* + 1) = g(d^*)$. Let $\beta : I \to \{0, \dots, r\}$ be such that $g(d^*) = \sum_{v \in I} \beta(v)$ and $E(d^*, \beta)$ holds. Let $\vec{V} \in \C^{d^*}_\beta$ be in the Scott ideal $\{X_c : c \in M\}$ and $b$ be such that $X_b = X_a[\vec{V}]$. Then $(I, b)$ is a $\PP$-precondition and is $f$-matching since only the reservoir is restricted. Moreover, Item (1) holds by definition of $\C^{d^*}_\beta$.

We claim that Item (2) holds. Suppose for the contradiction that $\beta(v) < r$ and $(v, b) \nqvdash_{\ell \cdot \card I} \varphi_{v, \beta(v)}(G)$ for some~$v \in I$. The tree 
$$
U = \{ u \in \T_{A,\ell \cdot \card I} : \forall \h u \in \OSub^=(u)  \neg \varphi_{v, \beta(v)}(v \cdot X_b[\h u])\}
$$ 
is $M$-unbounded. Let $V' \in [U]$
in the Scott ideal $\{X_c : c \in M\}$. Then $(v, X_b[V']) \Vdash \neg \varphi_{v, \beta(v)}(G)$. Let $\beta'$ be equal to $\beta$, except that $\beta'(v) = \beta(v)+1 \leq r$. Note that $(\vec{V}, V') \in \C^{d^*+1}_{\beta'}$, so $E(d^*+1, \beta')$ holds. Moreover, since $g$ is non-decreasing and $d^*$ is the least such that $g(d^*+1) = g(d^*)$, we have $d^* \leq r \cdot \card I$, so $d^* + 1$ belongs to the domain of~$g$. So
$$
g(d^*+1) \geq \sum_{w \in I} \beta'(w) = g(d^*)+1 > g(d^*)
$$
contradicting the choice of~$d^*$. So Item~(2) holds. This completes the proof of the claim.

Last, all the steps of the construction are effective. The table of~$E$ and the function~$g$ are computed from $M$-finitely many bits of~$(Q \oplus Y)' \leq_T Y'$. The search for~$d^*$ and $\beta$ is bounded. So $b$ and $\beta$ are computed by a $\Delta_1(Y')$ function.
This completes the proof of \Cref{lem:pi01-extension}.
\end{proof}

We now combine \Cref{lem:pi01-extension} and \Cref{lem:sigma01-extension} to prove a simultaneous decision lemma, which is the Shore blocking counterpart of \cite[Lemma 3.20]{liu2026reverse}. 

\begin{lemma}[Simultaneous decision]\label[lemma]{lem:simultaneous-decision}
Let $(I, a)$ be a $\PP$-condition, $r \in M$ and $(\varphi_{v, i})_{v \in I, i < r}$ be a uniformly $\Sigma_1(Y)$-formulas. There is some $\PP$-condition $(K, c) \leq (I, a)$ such that for every $v \in I$, every $\hat v \in K$ extending~$v$ and every~$i < r$, the $\PP$-branch $(\hat v, c)$ decides $\varphi_{v,i}(G)$. More precisely, letting $S_{\hat v} = \{ i < r : \varphi_{v,i}(\hat v) \}$,
\begin{itemize}
    \item[(1)] $(\h v, c) \Vdash \varphi_{v,i}(G)$ for every $i \in S_{\h v}$;
    \item[(2)] $(\h v, c) \Vdash \neg \varphi_{v,i}(G)$ for every $i \in r \setminus S_{\h v}$.
\end{itemize}
Moreover, $(K, c)$ can be computed from $(I, a)$, $r$ and $(\varphi_{v, i})_{v \in I, i < r}$ by a $\Delta_1(Y')$ function.
\end{lemma}
\begin{proof}
For every~$v \in I$ and $i \leq r$, let $\psi_{v, i}(G)$ be the $\Sigma_1(Y)$-formula
$$
\exists S \subseteq \{0, \dots, r-1\} (\card S = r-i \wedge \forall j \in S\ \varphi_{v,j}(G)).
$$
Note that the formulas $(\psi_{v,i})_{v \in I, i < r}$ satisfy the hypothesis of the $\Pi^0_1$-extension lemma (\Cref{lem:pi01-extension}), so there is a $\PP$-condition $(I, b) \leq (I, a)$ and $\beta : I \to \{0, \dots, r\}$ such that for every $v \in I$:
\begin{itemize}
    \item[(a)] if $\beta(v) > 0$, then $(v, b) \Vdash \neg \psi_{v,\beta(v)-1}(G)$;
    \item[(b)] if $\beta(v) < r$, then $(v, b) \qvdash_{\ell \cdot \card I} \psi_{v, \beta(v)}(G)$.
\end{itemize}
Let $J = \{ v \in I : \beta(v) < r \}$. By (b) and the $\Sigma^0_1$-extension lemma applied to $(I, b)$, $J$ and $(\psi_{v, \beta(v)})_{v \in J}$, there is some~$\PP$-condition $(K, c) \leq (I, b)$ such that $\psi_{v,\beta(v)}(\hat v)$ holds for every~$v \in J$ and every $\h v \in K$ extending~$v$.

Fix $v \in I$ and $\h v \in K$ extending~$v$. Item (1) holds by \Cref{caOVW-def-forcing}(1). By (b), letting $S_{\hat v} = \{ i < r : \varphi_{v,i}(\hat v) \}$, $\card S_{\h v} \geq r - \beta(v)$. Moreover, by (a), $\card S_{\h v} \leq r - \beta(v)$, so $\card S_{\h v} = r - \beta(v)$. Suppose for the contradiction that Item (2) fails, that is, suppose that $(\h v, c) \nvdash \neg \varphi_{v, i}(G)$ for some~$i \in r \setminus S_{\h v}$. Then there is some finite $w \in [\h v, c]$ such that $\varphi_{v, i}(w)$ holds. Since $\h v \preceq w$, then by monotonicity, $\varphi_{v, j}(w)$ holds for every $j \in S_{\h v} \cup \{i\}$. These are $r - \beta(v)+1$ formulas. In particular, $\beta(v) > 0$, and $w \in [v, b]$ contradicts (a).

The effectiveness follows from that of \Cref{lem:pi01-extension} and \Cref{lem:sigma01-extension}.
\end{proof}

\subsection{Forcing the requirements}

We are going to force two kind of requirements: 
\begin{itemize}
    \item[($\H$)] \emph{Homogeneity}: we want to ensure that for the generic $M$-ordered variable word~$G$ will be $f$-monochromatic at $M$-unboundedly many levels;
    \item[($\L$)] \emph{Weak lowness}: we want to decide initial segments of the jump of~$G \oplus Y$ simultaneously for segments of $M$-unbounded length.
\end{itemize}

Given some~$t \in M$, let $\varphi_t(G)$ be the following $\Sigma_1(Y)$-formula:
\begin{align}\label{caOVW-defvarphi}
&\text{ for some $n$, $G$ has at least $n+1$ variable kinds, } \nonumber\\
& \text{ the first occurrence of $x_n$ is at position at least~$t$, and } \nonumber\\
& \text{ $G[A^n] = \{ G[u] : u \in A^n \}$ is $f$-monochromatic.}
\end{align}

The following lemma is a straightforward adaptation of Liu and Patey~\cite[Lemma 3.21]{liu2026reverse} which holds over~$\RCA_0$:

\begin{lemma}[Liu and Patey~\cite{liu2026reverse}, $\RCA_0$]\label[lemma]{lem:deletion}
Let $(I, a)$ be a $\PP$-condition. Let $J$ be the set of $v \in I$ such that $(v, a) \Vdash \neg \varphi_{t(I)}(G)$. Then $(I \setminus J, a)$ is a $\PP$-condition.
\end{lemma}

We can therefore use the simultaneous decision lemma in a very general form to obtain the desired density lemma for homogeneity requirements:

\begin{lemma}[Homogeneity]\label[lemma]{lem:homogeneity-density}
Let $(I, a)$ be a $\PP$-condition and $r \in M$. There is a $\PP$-condition $(K, c) \leq (I, a)$ such that for every $\h v \in K$, $(\h v, c) \Vdash \varphi_r(G)$.
Moreover, $(K, c)$ is computed from $(I, a)$ and $r$ by a $\Delta_1(Y')$ function.
\end{lemma}
\begin{proof}
First, let $J = \{ v \cdot X_a[w] : w \in A^r \}$, and let $X_b$ be $X_a$ appropriately truncated, so that $(J, b) \leq (I, a)$ with $t(J) > r$. In particular, $(J, b)$ is $f$-matching, so it is a $\PP$-condition.
Apply \Cref{lem:simultaneous-decision} with $(J, b)$ and the family $(\varphi_{v,0})_{v \in J}$, where $\varphi_{v,0}(G) := \varphi_r(G)$. The $\PP$-branches which do not force $\varphi_r(G)$ force $\neg \varphi_r(G)$, hence force $\neg \varphi_{t(K)}(G)$ since $r \leq t(K)$. By \Cref{lem:deletion}, one can remove all the branches which do not force $\varphi_r(G)$.
\end{proof}

The following density lemma for weak lowness requirements exploits the full power of the simultaneous decision lemma:

\begin{lemma}[Deciding the jump]\label[lemma]{lem:jump-density}
Let $(I, a)$ be a $\PP$-condition and $r \in M$. There is a $\PP$-condition $(K, c) \leq (I, a)$ such that for every $\h v \in K$, letting $S_{\h v} = \{ e < r : \Phi_e^{\h v \oplus Y} (e)\downarrow \}$, we have for every $e < r$:
\begin{itemize}
    \item[(a)] if $e \in S_{\h v}$, then  $(\h v, c) \Vdash \Phi_e^{G \oplus Y}(e)\downarrow$;
    \item[(b)] if $e \not\in S_{\h v}$, then $(\h v, c) \Vdash \Phi_e^{G \oplus Y}(e)\uparrow$.
\end{itemize}
Moreover, $(K, c)$ is computed from $(I, a)$ and $r$ by a $\Delta_1(Y')$ function.
\end{lemma}
\begin{proof}
Apply \Cref{lem:simultaneous-decision} with $\varphi_{v,e}(G) := \Phi_e^{G \oplus Y}(e)\downarrow$ for every~$v \in I$.
\end{proof}

\subsection{Proof of \Cref{prop:ovw-step}}

We now turn to the proof of \Cref{prop:ovw-step}.
Construct the following $Y'$-computable decreasing sequence of $\PP$-conditions 
$$(I_0, a_0) \geq (I_1, a_1) \geq \cdots$$
by stages as follows. First, $I_0 = \{\epsilon\}$, where $\epsilon$ is the empty sequence, and $X_{a_0}$ is the ordered $M$-variable word $x_0x_1x_2\cdots$. Assume $(I_s, a_s)$ has been defined. Let $r(s) = \max \{s, t(I_s), a_s\}$, where $t(I_s)$ is the common length of the elements of~$I_s$. Using \Cref{lem:homogeneity-density} and \Cref{lem:jump-density}, one can $Y'$-computably find a $\PP$-condition $(I_{s+1}, a_{s+1})$ such that
\begin{itemize}
    \item[(1)] for every~$v \in I_{s+1}$, $(v, a_{s+1}) \Vdash \varphi_{r(s)}(G)$;
    \item[(2)] for every~$v \in I_{s+1}$ and every $e < r(s)$, $(v, a_{s+1})$ decides $\Phi^{G \oplus Y}_e(e)$.
\end{itemize}
The set $J$ of all stages~$s$ such that $(I_s, a_s)$ is defined forms a $\Sigma_1(Y')$ cut in~$M$, which might be proper. Note that by $\Pi^0_1$-induction in~$\M[Q]$, for every~$s < t \in J$, $(I_s, a_s) \geq (I_t, a_t)$.

\begin{claim}
The sets $R = \{ r(s) : s \in J \}$ and $L = \{ t(I_s) : s \in J \}$ are $M$-unbounded.
\end{claim}
\begin{proof}
If $J = M$, then since $r(s) \geq s$, $R$ is $M$-unbounded.
If $J \neq M$, then if $R$ were $M$-bounded, $J$ would be a proper $\Delta_1(Y')$ cut, contradicting the fact that $\M[Y'] \models \IDelta_1$. 
Last, by (1), $t(I_{s+1}) \geq r(s)$, so $L$ is $M$-unbounded.
\end{proof}
% Moreover, the set $\{ r(s) : s \in J \}$ is $M$-unbounded, otherwise $J$ would be a proper $\Delta_1(Y')$ cut, contradicting the fact that $\M[Y'] \models \IDelta_1$. It follows by (1) that the set $\{ t(I_s) : s \in J \}$ is also $M$-unbounded.

Let $\TT$ be the downward-closure of $\bigcup_{s \in J} I_s$.

\begin{claim}
$\TT$ is an $M$-unbounded, $Y'$-computably bounded $\Delta_1(Y')$ tree.
\end{claim}
\begin{proof}
$\TT$ is $M$-unbounded since $L$ is $M$-unbounded.
Moreover, by length-homogeneity of the set of stems of a $\PP$-precondition and the definition of the $\PP$-precondition extension relation, $\TT$ is $Y'$-computably bounded and $\Delta_1(Y')$: to know whether some node~$w$ belongs to~$\TT$, search $Y'$-computably for the least stage~$s$ such that $t(I_s) \geq |w|$, and check whether $w$ is an initial segment of any element of~$I_s$.
\end{proof}

Since $\M[P] \models \RCA_0^*$ and $P \gg Y'$, by \Cref{lem:wkl-to-bounded}, $P$ computes an $M$-infinite path~$G \in [\TT]$. Note that for every~$s$, $G \in [I_s, a_s]$. By (1), there are $M$-unboundedly many levels~$n$ such that $G[A^n] = \{ G[v] : v \in A^n \}$ is $f$-monochromatic. 

\begin{claim}
$(G \oplus Y)' \leq_T P$ and $\M[G] \models \RCA_0 + \BSig_2$.
\end{claim}
\begin{proof}
Given~$e \in M$, to decide whether $\Phi^{G \oplus Y}_e(e)\downarrow$, $P$-computably search for some~$s \in J$ such that $r(s) > e$. Let $v \in I_{s+1}$ be such that $v \prec G$. By (2), $\Phi^{G \oplus Y}_e(e)\downarrow$ iff $\Phi^{v \oplus Y}_e(e)\downarrow$.
The second part of the claim follows from \Cref{lem:jump-bsig2-bsig1}.
\end{proof}

By the infinite pigeonhole principle ($\RT^1$) which holds over~$\RCA_0 + \BSig_2$ (see Hirst~\cite[Theorem 6.4]{hirst1987combinatorics}), there is some color~$\ell^* < \ell$ and some $M$-unbounded set~$H = \{ h_0 < h_1 < \dots \} \subseteq M$ in~$\M[G]$ such that $G[A^n]$ is $f$-monochromatic for color~$\ell^*$ for every~$n \in H$.
Let $\h G = G[V]$ where $V(i) \in A$ if $i \not \in H$ and $V(i) = x_t$ if $i = h_t$. Then $\WSub^*(\h G)$ is $f$-monochromatic for color~$\ell^*$. Moreover, $(\h G \oplus Y)' \leq_T (G \oplus Y)' \leq_T P$. This completes the proof of \Cref{prop:ovw-step}.

\section{Open questions}

The Ordered Variable Word theorem can be weakened to allow the variable kind to be interleaved, and occur infinitely often. This is known as the \emph{Carlson-Simpson Lemma} $\CSL$. Its zero-dimensional version ($\CSL^0$) is an immediate consequence of $\OVW^0$, and is therefore $\Pi^1_1$-conservative over~$\RCA_0 + \BSig_2$.

The one-dimensional version for $\ell$ colors, $\CSL^1_\ell$, is a generalization of the tree theorem for pairs and $\ell$ colors ($\mathsf{TT}^2_\ell$), which itself generalizes Ramsey's theorem for pairs and $\ell$ colors ($\RT^2_\ell$). Hwang~\cite{hwang2026conservation} proved that $\RCA_0 + \RT^2_2$ is a $\Pi^1_1$-conservative extension of~$\RCA_0 + \BSig_2$, and Le Houérou and Patey~\cite{houerou2026csl} proved that $\RCA_0 + \CSL^1_2$ is a $\forall \Pi^0_4$-conservative extension of~$\RCA_0 + \BSig_2$. 

\begin{question}
Is $\RCA_0 + \CSL^1_2$ $\Pi^1_1$-conservative over~$\RCA_0 + \BSig_2$?
\end{question}

\begin{question}
Is $\RCA_0 + \mathsf{TT}^2_2$ $\Pi^1_1$-conservative over~$\RCA_0 + \BSig_2$?
\end{question}

\subsection*{AI disclaimer}

The original idea of formalizing the weakly low basis theorem for $\OVW^0$ from Liu and Patey~\cite{liu2026reverse} using the techniques of Le Houérou, Patey and Yokoyama~\cite{houerou2025wf} came from the author. Claude Opus 5.5 (Anthropic) was used to prove the $\Pi^0_1$-extension lemma (\Cref{lem:pi01-extension}) which is a core lemma in the proof. The proof was completely rewritten by the author for clarity. 

% \begin{center}
% \textbf{Acknowledgement}
% \end{center}
% The authors are thankful to ?? for insightful comments and discussions.
 
\bibliographystyle{plain}
\bibliography{biblio.bib}

\end{document}